\documentclass[11pt]{article}
\usepackage[utf8]{inputenc}
\usepackage{fullpage}
\usepackage{amsmath,amsthm,amsfonts,dsfont}
\usepackage{amssymb,latexsym,graphicx}
\usepackage{mathtools}
\usepackage{hyperref}
\usepackage{xcolor}
\hypersetup{
	colorlinks,
	linkcolor={red!75!black},
	citecolor={blue!75!black},
	urlcolor={red!75!black}
}

\newtheorem{theorem}{Theorem}[section]

\newtheorem{proposition}[theorem]{Proposition}
\newtheorem{lemma}[theorem]{Lemma}

\newtheorem{claim}[theorem]{Claim}

\newtheorem{corollary}[theorem]{Corollary}

\newcommand{\C}{\ensuremath{\mathbb{C}}}

\newcommand{\R}{\ensuremath{\mathbb{R}}}
\newcommand{\Z}{\ensuremath{\mathbb{Z}}}

\newcommand{\lat}{\mathcal{L}}

\newcommand{\eps}{\varepsilon} 
\renewcommand{\epsilon}{\varepsilon}

\newcommand{\rank}{\mathrm{rank}}
\newcommand{\vol}{\mathrm{vol}}

\DeclareMathOperator{\spn}{span}

\renewcommand{\vec}[1]{\ensuremath{\boldsymbol{#1}}}
\newcommand{\basis}{\ensuremath{\mathbf{B}}}

\DeclarePairedDelimiter\inner{\langle}{\rangle}
\DeclarePairedDelimiter\floor{\lfloor}{\rfloor}

\mathtoolsset{centercolon}\usepackage{enumitem}
\usepackage{mathrsfs}
\usepackage[hyperpageref]{backref}

\title{A Tight Reverse Minkowski Inequality for the Epstein Zeta Function}
\author{Yael Eisenberg\thanks{Cornell University. This material is based upon work supported by the National Science Foundation Graduate Research Fellowship under Grant No. DGE - 1650441.} 
\and 
Oded Regev\thanks{Courant Institute of Mathematical Sciences, New York
				University. Supported by the Simons Collaboration on Algorithms and Geometry, by a Simons Investigator Award from the Simons Foundation, and by the National Science Foundation (NSF) under Grant No.~CCF-1320188. Any opinions, findings, and conclusions or recommendations expressed in this material are those of the authors and do not necessarily reflect the views of the NSF.}
\and 
Noah Stephens-Davidowitz\thanks{Cornell University. Supported in part by the National Science Foundation under Grant No.~CCF-2122230.}}
\date{}

\newcommand{\zf}[2]{\zeta(#1,#2)}
\newcommand{\qzf}[3]{\zeta(#1,#2;#3)}
\newcommand{\qzfp}[3]{\zeta'(#1,#2;#3)}
\newcommand{\Kbar}{\overline{K}}
\newcommand{\intd}{\,{\rm d}}

\begin{document}

\maketitle

\begin{abstract}
    We prove that if $\lat \subset \R^n$ is a lattice such that $\det(\lat') \geq 1$ for all sublattices $\lat' \subseteq \lat$, then
    \[
        \sum_{\substack{\vec{y}\in\lat\\\vec{y}\neq\vec0}} (\|\vec{y}\|^2+q)^{-s} \leq \sum_{\substack{\vec{z} \in \Z^n\\\vec{z}\neq\vec0}} (\|\vec{z}\|^2+q)^{-s}
    \]
    for all $s > n/2$ and all $0 \leq q \leq (2s-n)/(n+2)$, with equality if and only if $\lat$ is isomorphic to $\Z^n$. 
\end{abstract}

\section{Introduction}

 A lattice $\lat \subset \R^n$ is the set of all integer linear combinations of linearly independent basis vectors $\basis = (\vec{b}_1,\ldots, \vec{b}_n)$, i.e.,
 \[
    \lat = \{z_1 \vec{b}_1 + \cdots + z_n \vec{b}_n \ : \ z_i \in \Z\}
    \; .
 \]
 The determinant of the lattice, $\det(\lat) = |\det(\basis)|$, is a measure of its density on large scales, in the sense that
 \[
 \det(\lat) = \lim_{r \rightarrow \infty} \frac{\vol(r B_2^n)}{|\lat \cap r B_2^n|}
 \; ,
 \]
 where $rB_2^n$ denotes the closed Euclidean ball of radius $r >0$, whose volume is 
 roughly $(2\pi er^2/n)^{n/2}$. Minkowski's celebrated first theorem relates this large-scale density to smaller-scale density. Specifically, the following variant of Minkowski's theorem due to Blichfeldt and van der Corput shows that a lattice that is dense over large scales must also contain many points in a fixed ball.\footnote{Blichfeldt and van der Corput actually showed the slightly stronger bound $|\lat \cap rB_2^n| \geq 2\floor{2^{-n} \cdot \vol(r B_2^n)}+1$ and considered arbitrary norms, not just $\ell_2$.  (See, e.g.,~\cite[Thm.~1 of Ch.~2, Sec.~7]{Lekkerkerker_book}.)}
 
 \begin{theorem}[\cite{vanderCorput1936}]
			\label{thm:minkowski}
			For any lattice $\lat \subset \R^n$ with $\det(\lat) = 1$ and $r > 0$,
			\begin{equation*}
			|\lat \cap r B_2^n| 
			\geq 2^{-n} \cdot \vol(r B_2^n) 
			\approx \Big(\frac{\pi er^2}{2n}\Big)^{n/2}
			\; . 
			\end{equation*}
\end{theorem}

In~\cite{regevReverseMinkowskiTheorem2017}, we proved a partial converse to Theorem~\ref{thm:minkowski}, which was originally conjectured by Dadush~\cite{priv:Daniel,DR16}. In particular, we showed that any lattice $\lat \subset \R^n$ with many short points must have a \emph{sublattice} $\lat' \subseteq \lat$ with small determinant.\footnote{Notice that we cannot hope for the lattice itself to have small determinant, since a lattice can have large determinant but contain a dense sublattice. E.g., consider the lattice $\lat \subset \R^2$ generated by the vectors $(1/T, 0)$ and $(0,T^2)$ for arbitrarily large $T \gg 1$, which has arbitrarily large determinant $\det(\lat) = T$ but arbitrarily many short points $|\lat \cap r B_2^2| \geq 1 + 2\floor{rT}$.} We refer to results of this flavor as ``reverse Minkowski theorems.'' Here is the main theorem from~\cite{regevReverseMinkowskiTheorem2017} (stated in the contrapositive).

\begin{theorem}[\cite{regevReverseMinkowskiTheorem2017}]
    \label{thm:RM}
    For any lattice $\lat \subset \R^n$ with $\det(\lat') \geq 1$ for all sublattices $\lat' \subseteq \lat$,
    \[
        \Theta(\lat,i\tau) \leq 3/2
        \; ,
    \]
    where the lattice theta function is defined as 
    \[
        \Theta(\lat,i\tau) := \sum_{\vec{y} \in \lat} \exp(-\pi \tau \|\vec{y}\|^2)
        \; ,
    \]
    and $\tau \geq 100(\log n + 2)^2$.
\end{theorem}

In particular, for any lattice $\lat \subset \R^n$ as above and any $r > 0$, $|\lat \cap rB_2^n| \leq 3\exp(\pi \tau r^2)/2$.
Many other applications are shown in~\cite{DR16,regevReverseMinkowskiTheorem2017}. (See also~\cite{Bost20}.)

However, Theorem~\ref{thm:RM} seems unlikely to be tight. As far as we know, the worst case is the integer lattice $\lat = \Z^n$, which satisfies $\Theta(\Z^n,i\tau) = 3/2$ for $\tau = \log n/\pi + O(1)$. Indeed, the lattice $\Z^n$ arises naturally in this context, and a major open question raised in~\cite{regevReverseMinkowskiTheorem2017} is whether the integer lattice $\Z^n$ is actually the worst case for \emph{all} $\tau > 0$. In other words, is it the case that
\begin{equation}
    \label{eq:tight_RM}
    \Theta(\lat, i\tau) \leq \Theta(\Z^n,i\tau) 
\end{equation}
for all $\tau > 0$ and all lattices $\lat \subset \R^n$ such that all sublattices $\lat' \subseteq \lat$ have $\det(\lat') \geq 1$? We call Eq.~\eqref{eq:tight_RM} a \emph{tight} reverse Minkowski inequality. 

If we could prove Eq.~\eqref{eq:tight_RM} for all $\tau > 0$, we would immediately obtain similar inequalities for many other functions of interest. In this work, we are interested in the Epstein zeta function, defined as
\[
    \zf{\lat}{s} := \sum_{\substack{\vec{y} \in \lat\\ \vec{y} \neq \vec0}} \|\vec{y}\|^{-2s}
\]
for $s > n/2$. This function arises in a number of contexts (see, e.g.,~\cite{SS06}), and it can be written as
\begin{equation}
    \label{eq:zeta_from_theta_no_q}
    \zf{\lat}{s} = \frac{\pi^s}{\Gamma(s)} \int_0^\infty \tau^{s-1} (\Theta(\lat,i\tau)-1) \intd \tau
    \; .
\end{equation}
So, if we could prove Eq.~\eqref{eq:tight_RM} for all $\tau > 0$, then we would also obtain a proof that
\begin{equation}
    \label{eq:tight_zeta_no_q}
    \zf{\lat}{s} \leq \zf{\Z^n}{s}
    \; .
\end{equation}
We can therefore view a proof of Eq.~\eqref{eq:tight_zeta_no_q} as partial progress towards proving Eq.~\eqref{eq:tight_RM}. (We say more about the relationship between $\zeta$ and $\Theta$ in Section~\ref{sec:zeta_and_theta}.) 

We are also interested in the function
\[
    \qzf{\lat}{s}{q} := \sum_{\vec{y} \in \lat} (\|\vec{y}\|^2 + q)^{-s}
\]
for $q > 0$ and $s > n/2$. 
Notice that this can be thought of as a generalization of $\zf{\lat}{s}$, as 
\begin{equation}
    \label{eq:q_to_zero}
    \zf{\lat}{s} =  \lim_{q \to 0} \qzf{\lat}{s}{q} - q^{-s}
    \; .
\end{equation}
To conveniently handle both cases simultaneously, we will work with 
\[
    \qzfp{\lat}{s}{q} := \sum_{\substack{\vec{y} \in \lat\\ \vec{y} \neq \vec0}} (\|\vec{y}\|^2 + q)^{-s}
    \; 
\]
for any $q \geq 0$.
Notice that $\zf{\lat}{s} = \qzfp{\lat}{s}{0}$ and $\qzf{\lat}{s}{q} = \qzfp{\lat}{s}{q} + q^{-s}$ for $q>0$.

Our main result is then the following tight reverse Minkowski theorem for  $\qzfp{\lat}{s}{q}$. We also characterize the case of equality, showing that $\qzfp{\lat}{s}{q} = \qzfp{\Z^n}{s}{q}$ (for lattices $\lat$ with no dense sublattices) if and only if $\lat$ is isomorphic to $\Z^n$, i.e., if and only if $\lat = R\Z^n$ for some orthogonal transformation $R \in \mathsf{O}_n(\R)$.

\begin{theorem}
    \label{thm:zeta_RM_intro}
    For any lattice $\lat \subset \R^n$ with $\det(\lat') \geq 1$ for all sublattices $\lat' \subseteq \lat$, any $s > n/2$, and $0 \leq q \leq (2s-n)/(n+2)$, 
    \[
        \qzfp{\lat}{s}{q} \leq \qzfp{\Z^n}{s}{q}
        \; ,
    \]
    with equality if and only if $\lat$ is isomorphic to $\Z^n$.
\end{theorem}

\subsection{Our techniques}

In order to prove Theorem~\ref{thm:zeta_RM_intro}, we first restrict our attention to \emph{stable} lattices, which are lattices $\lat \subset \R^n$ such that $\det(\lat) = 1$ and $\det(\lat') \geq 1$ for all sublattices $\lat' \subseteq \lat$. Theorem~\ref{thm:zeta_RM_intro} follows easily once we have proven it for the special case of stable lattices. (See Item~\ref{item:contraction} of Proposition~\ref{prop:stable}.) The set of stable lattices is a \emph{compact} subset of the set of all lattices (under the quotient topology of $\mathsf{SL}_{n}(\R)/\mathsf{SL}_{n}(\Z)$), so that $\qzfp{\lat}{s}{q}$ must achieve its maximum over this set. It therefore suffices to assume that our lattice $\lat \subset \R^n$ maximizes $\qzfp{\lat}{s}{q}$ over the set of stable lattices, and then to prove that any such maximizer must equal $\Z^n$ (up to isomorphism, but we ignore this distinction in this high-level discussion).

Indeed, $\qzfp{\lat}{s}{q}$ is a particularly nice function for this purpose because it has positive Laplacian for $q$ sufficiently small. (In fact, it is known that $\zf{\lat}{s}$
    is an eigenfunction of the Laplacian, with eigenvalue $s(1-1/n)(s-n/2) > 0$~\cite{SS06}.)
    In other words, the expectation of $\qzfp{\lat'}{s}{q}$ for a small random perturbation $\lat'$ of $\lat$ is larger than $\qzfp{\lat}{s}{q}$.
    This would contradict the assumption that $\lat$ maximizes $\zeta'$ over the set of stable lattices \emph{unless} $\lat$ lies on the \emph{boundary} of this set. So, $\lat$ must lie on the boundary of the set of stable lattices, i.e., $\lat$ must have a non-trivial sublattice $\lat' \subset \lat$ such that $\det(\lat') = 1$. 
    In fact, we can show (using the fact that the Fourier transform of
    $\vec{y} \mapsto (\|\vec{y}\|^2 + q)^{-s}$
     is non-negative) that our global maximizer $\lat$ must actually be a direct sum $\lat = \lat' \oplus \lat''$.
    
     In other words, any global maximizer $\lat$ must share a key property with $\Z^n$, in that it must be expressible as a non-trivial direct sum $\lat = \lat' \oplus \lat''$. To prove that $\lat$ must actually be equal to $\Z^n$, we argue that for any such decomposition $\lat = \lat' \oplus \lat''$ of our global maximizer, either (1) $\lat''$ has rank one, or (2) $\lat''$ can itself be written as a direct sum of lower-dimensional lattices. This is sufficient to prove that $\lat$ is $\Z^n$. 
     To show this, we use an argument like the one above. In particular, it suffices to prove that the function $\qzfp{\lat' \oplus \lat''}{s}{q}$ has a positive Laplacian, as a function of $\lat''$ (provided that $\rank(\lat'') > 1$). 
     To that end, using the Poisson Summation Formula (Theorem~\ref{thm:PSF}), we write this Laplacian as a linear combination of modified Bessel functions of the second kind $K_{\alpha}(x)$ with different parameters $\alpha$. We then apply a recurrence relation, which relates $K_{\alpha}(x)$ to $K_{\alpha-1}(x)$ (Lemma~\ref{lem:zeta_recurrence_inequality}) and in particular implies that this Laplacian is positive.

\subsection{Relationship with the theta function and potential ways forward}
\label{sec:zeta_and_theta}

As mentioned above, we view Theorem~\ref{thm:zeta_RM_intro} as progress towards proving a tight reverse Minkowski theorem like Eq.~\eqref{eq:tight_RM} for all $\tau > 0$. Indeed, generalizing Theorem~\ref{thm:zeta_RM_intro} to all $q \geq 0$ (as opposed to $0 \leq q \leq (2s-n)/(n+2)$) is equivalent to proving such a result. This follows immediately from the identities
\begin{equation}
    \label{eq:zeta_from_theta}
    \qzf{\lat}{s}{q} = \frac{\pi^s}{\Gamma(s)} \int_0^\infty \tau^{s-1}\exp(-\pi q \tau) \Theta(\lat,i\tau) \intd \tau
    \; ,
\end{equation}
(see Eq.~\eqref{eq:zeta_from_theta_pointwise} below) and
\begin{equation}
    \label{eq:theta_from_zeta}
    \Theta(\lat,i\tau) = \lim_{s \to \infty} (\pi \tau/s)^{-s} \qzf{\lat}{s}{s/(\pi \tau)}
    \; .
\end{equation}

In fact, Eq.~\eqref{eq:theta_from_zeta} together with Theorem~\ref{thm:zeta_RM_intro} immediately shows that
\[
    \Theta(\lat,i\tau) \leq \Theta(\Z^n,i\tau)
\]
    for any lattice $\lat \subset \R^n$ with $\det(\lat') \geq 1$ for all sublattices $\lat' \subseteq \lat$ and any  $\tau \geq (n+2)/(2\pi)$. This exact result was already proven in~\cite{regevReverseMinkowskiTheorem2017} (using similar techniques applied directly to $\Theta(\lat,i \tau)$), but the above gives an alternative proof.

However, it seems that our techniques cannot be directly extended to prove Eq.~\eqref{eq:tight_RM} for arbitrary $\tau > 0$, or equivalently, to prove Theorem~\ref{thm:zeta_RM_intro} for all $q \geq 0$. Indeed,  Heimendahl et al.\ recently showed that there exist $\tau > 0$ and stable lattices $\lat \subset \R^n$ that maximize $\Theta(\lat',i\tau)$ in a neighborhood of $\lat$~\cite{HeimendahlMTVZ21}. Since our proof crucially relies on our ability to rule out such local maxima, we cannot hope to use similar techniques to prove Eq.~\eqref{eq:tight_RM} for such $\tau$. 

The proof techniques in~\cite{regevReverseMinkowskiTheorem2017} (which were originally used by Shapira and Weiss in a slightly different context~\cite{SW16}) do offer a potential way forward. The idea is to characterize any such local maxima $\lat$ and to use this characterization to directly bound $\Theta(\lat,i\tau)$. In fact, the local maxima found in~\cite{HeimendahlMTVZ21} are very special lattices with rich mathematical structure (they are the $32$-dimensional even unimodular lattices with no roots, and they share the same theta function, which is a modular form). Perhaps one can show that some properties of these special lattices must be shared by any local maximum $\lat$, and then use this to directly bound $\Theta(\lat,i\tau)$.

\subsection*{Acknowledgments}

We thank Stephen D. Miller and Barak Weiss for helpful discussions.

\section{Preliminaries}

\subsection{Lattices}
A \emph{lattice} $\lat \subset \R^n$ of rank $d$ is the set of integer linear combinations of linearly independent basis vectors $\basis := (\vec{b}_1,\ldots, \vec{b}_d) \in \R^{n \times d}$,
\[
\lat = \lat(\basis) := \Big\{ \sum_{i=1}^d a_i \vec{b}_i \ : \ a_i \in \Z \Big\}
\; .
\]
We typically implicitly treat lattices as though they are full rank (i.e., $d = n$) by implicitly identifying $\spn(\lat)$ with $\R^d$.
The \emph{dual lattice}
\[
\lat^* := \{ \vec{w} \in \spn(\lat) \ : \ \forall \vec{y} \in \lat,\ \inner{\vec{w}, \vec{y}} \in \Z\}
\]
is the set of all vectors in the span of $\lat$ that have integer inner products with all lattice vectors. 
One can check that $\lat^{**} = \lat$ and that $\lat^*$ is generated by the basis $\basis^* := \basis (\basis^T \basis)^{-1}$.

We say that lattices $\lat_1, \lat_2 \subset \R^n$ are \emph{isomorphic} if there exists an orthogonal transformation $R\in \mathsf{O}_n(\R)$ such that $R\lat_2 = \lat_1$. We write this as $\lat_1 \cong \lat_2$.

We write
\[
\lambda_1(\lat) := \min_{\vec{y} \in \lat \setminus \{\vec0\}} \|\vec{y}\|
\]
for the length of the shortest non-zero lattice vector.

The \emph{determinant} of the lattice is given by $\det(\lat) := \sqrt{\det(\basis^T \basis)}$, or simply $|\det(\basis)|$ in the full-rank case. We also define $\det(\{\vec0\}) := 1$ for convenience. One can show that the determinant is well defined (i.e., it does not depend on the choice of basis $\basis$). It follows that if $\lat \subset \R^n$ is a full-rank lattice and $A \in \R^{n\times n}$ is non-singular, then $\det(A \lat) = |\det(A)| \det(\lat)$, and that $\det(\lat^*) = 1/\det(\lat)$.

A \emph{sublattice} $\lat' \subseteq \lat$ is an additive subgroup of $\lat$. We say that $\lat'$ is \emph{primitive} if $\lat' = \lat \cap \spn(\lat')$. 
For a primitive sublattice $\lat' \subseteq \lat$, we define the quotient lattice $\lat/\lat' := \pi_{\lat^{\prime \perp}}(\lat)$ to be the projection of $\lat$ onto the space orthogonal to $\spn(\lat')$. In particular, $\lat/\lat'$ is a lattice with $\rank(\lat/\lat') = \rank(\lat) - \rank(\lat')$, and we have the identities $(\lat/\lat')^* = \lat^* \cap \spn(\lat')^\perp$ and $\det(\lat/\lat') = \det(\lat)/\det(\lat')$. 

We will need the following theorem, which is a version of the Poisson Summation Formula.

\begin{theorem}[{\cite[Theorem A.1]{millerKissingNumbersTransference2019}}]
\label{thm:PSF}
For any $f : \R^n \to \C$ with Fourier transform $\widehat{f}$ and any lattice $\lat \subset \R^n$
    \[\sum_{\vec{y} \in \lat} f(\vec{y}) = \frac{1}{\det(\lat)} \cdot \sum_{\vec{w} \in \lat^*} \widehat{f}(\vec{w})
    \; ,
    \]
    provided that (1) $f$ is continuous; (2) $|f(\vec{x})| = O((1+\|\vec{x}\|)^{-n-\delta})$ for some $\delta > 0$; and (3) the right-hand side converges absolutely.
\end{theorem}

\subsection{Stable lattices}

We say that a lattice $\lat \subset \R^n$ is \emph{stable} if $\det(\lat) = 1$ and $\det(\lat') \geq 1$ for all sublattices $\lat' \subseteq \lat$.
Next, we describe some useful properties of stable lattices.

\begin{proposition}\label{prop:stable}
 The following all hold.
\begin{enumerate}
	\item \label{item:stable_compact} The set of all stable lattices is compact (under the quotient topology of $\mathsf{SL}_n(\R)/\mathsf{SL}_n(\Z)$). 
	\item \label{item:stable_direct_sum} The direct sum of stable lattices is stable.
	\item \label{item:stable_quotient_boundary} A lattice $\lat \subset \R^n$ is on the boundary of the set of stable lattices if and only if $\lat$ is stable and there is a primitive sublattice $\lat' \subset \lat$ with $0 < \rank(\lat') < n$ such that $\lat'$ and $\lat/\lat'$ are both stable.
	\item \label{item:contraction} If $\lat \subset \R^n$ is a full-rank lattice such that $\det(\lat') \geq 1$ for all sublattices $\lat' \subseteq \lat$, then there exists a contracting linear transformation $A$ (i.e., a linear transformation such that $\|A\vec{x}\| \leq \|\vec{x}\|$ for all $\vec{x} \in \R^n$) such that $A\lat$ is stable. Furthermore, if $\det(\lat) > 1$, then there exists $\vec{y} \in \lat$ such that $\|A\vec{y}\| < \|\vec{y}\|$.
\end{enumerate}
    
\end{proposition}
\begin{proof}
    Items~\ref{item:stable_compact},~\ref{item:stable_direct_sum}, and~\ref{item:stable_quotient_boundary} are all taken directly from~\cite[Proposition 2.5]{regevReverseMinkowskiTheorem2017}. 
    Item~\ref{item:contraction} is used implicitly  throughout~\cite{regevReverseMinkowskiTheorem2017} (as was originally observed by Dadush~\cite{priv:Daniel17}). We therefore only provide a brief proof sketch of Item~\ref{item:contraction}.
    
   Let $\lat \subset \R^n$ be a lattice with $\det(\lat') \geq 1$ for all sublattices $\lat' \subseteq \lat$. Let $\lat_1 \subset \lat$ be a sublattice with maximal rank $k$ such that $\det(\lat_1) = 1$. (If no non-zero sublattice with determinant one exists, we take $\lat_1 := \{\vec0\}$ and $k = 0$, and we adopt the convention that $\det(\{\vec0\}) = 1$.) If $k = n$, then $\lat$ is stable and we are done.
   Otherwise, it suffices to show how to apply a contracting linear transformation $A$ to $\lat$ such that $A\lat$ still contains no sublattice with determinant less than one, but $A\lat$ contains a sublattice with determinant exactly one whose rank is strictly larger than $k$. (After applying this procedure at most $n$ times, the resulting lattice will be stable.)

    To show this, let $\lat_2 \subseteq \lat$ be the strict superlattice of $\lat_1$ (i.e., $\lat_1 \subset \lat_2$) that minimizes the function $
    \lat' \mapsto \det(\lat')^{1/\rank(\lat'/\lat_1)}$. (It follows from the fact that $\lat$ is discrete that this function achieves its minimum.) We will take $A$ to be the linear transformation that acts as the identity on $\spn(\lat_1)$ but scales $\spn(\lat_1)^\perp$ down by a factor of $\det(\lat_2)^{1/\rank(\lat_2/\lat_1)} > 1$. 
    
    By construction, we have that $A$ is a contraction with $\det(A\lat_2) = 1$. So, we only need to show that for any $\lat' \subseteq \lat$, $\det(A\lat') \geq 1$. By~\cite[Lemma 2.4]{regevReverseMinkowskiTheorem2017}, we have that 
    \[
        \det(A(\lat' + \lat_1))\det(A(\lat' \cap \lat_1)) \leq \det(A\lat_1) \det(A\lat')
        \; .
        \]
        Since $A$ acts as identity on $\spn(\lat_1)$, we see that $\det(A(\lat' \cap \lat_1)) = \det(\lat' \cap \lat_1) \geq 1$ and $\det(A\lat_1) = \det(\lat_1) = 1$. Moreover, by our choice of $\lat_2$ and $A$, we must have that $\det(A\lat'') \geq 1$ for any $\lat''$ with $\lat_1 \subseteq \lat''$. In particular, $\det(A(\lat' + \lat_1)) \geq 1$. We conclude that $\det(A\lat') \geq 1$, and the result follows.
\end{proof}

\subsection{Decompositions of lattices}
We call a lattice \emph{decomposable} if it is isomorphic to the direct sum of  non-trivial lattices with \emph{strictly} lower rank. Otherwise, we say it is \emph{indecomposable}. Notice that we may always write $\lat \cong \lat_1 \oplus \lat_2$ where $\lat_2$ is indecomposable (where $\lat_1$ might be the trivial lattice $\{\vec0\}$), as in the following claim.

\begin{claim}\label{claim:decompose}
For any non-trivial lattice $\lat\subset\R^n$ with rank $k$, there exist lattices $\lat_1,\lat_2$ such that $\lat\cong\lat_1\oplus\lat_2$ where $\lat_2$ is non-trivial and indecomposable (and $\lat_1$ could be the trivial lattice $\{\vec0\}$). If all such decompositions of a lattice $\lat\cong\lat_1\oplus\lat_2$ have $\rank(\lat_2)=1$, then $\lat\cong \alpha_1\Z\oplus\cdots\oplus\alpha_k \Z$ for some $\alpha_i\in\R$.
\end{claim}

\begin{proof}
The proof is by induction on $k$. Let $\lat\subset\R^n$ be a non-trivial lattice with rank $k$. Observe that the claim is true for $k=1$. Suppose that the result holds for any rank $m$ lattice $\lat'$ for $m < k$. If $\lat$ is indecomposable, then $\lat\cong\{\vec0\}\oplus\lat$. Otherwise, we have $\lat\cong \lat'_1\oplus\lat'_2$, where $1 \leq \rank(\lat_2') < k$. By the induction hypothesis, $\lat'_2\cong\lat''\oplus\lat_2$ where $\lat_2$ is indecomposable. Let $\lat_1=\lat'_1\oplus \lat''$, and we obtain $\lat\cong\lat_1\oplus\lat_2$ where $\lat_2$ is indecomposable, as needed.

Next, we assume that all decompositions $\lat \cong \lat_1 \oplus \lat_2$ with $\lat_2$ indecomposable have $\rank(\lat_2)= 1$ and show that this implies that $\lat\cong \alpha_1\Z\oplus\cdots\oplus\alpha_k \Z$. To that end, we decompose $\lat$ into $\lat\cong\lat'_1\oplus\lat'_2$, and repeatedly decompose the sublattices until we have $\lat\cong\lat'_1\oplus\lat'_2\oplus\cdots\oplus\lat'_\ell$, where $\lat_i'$ is indecomposable (and non-trivial) for all $i$.  It remains to show that $\rank(\lat_i') = 1$ for all $i$, i.e., that $\lat\cong \alpha_1\Z\oplus\cdots\oplus\alpha_k \Z$. Suppose the rank of $\lat'_i$ is greater than one for some $i$. Then, setting $\lat_2=\lat'_i$ and $\lat_1=\bigoplus_{1\leq j\leq \ell, j\neq i}\lat'_j$, we see that $\lat\cong\lat_1\oplus\lat_2$. However, since $\rank(\lat_2)>1$ and $\lat_2$ is indecomposable, this is a contradiction. We conclude that $\rank(\lat_i') = 1$ for all $i$ and therefore $\lat\cong \alpha_1\Z\oplus\cdots\oplus\alpha_k \Z$.
\end{proof}

\begin{corollary}\label{cor:decomposeZn}
    If $\lat\subset\R^n$ is a stable full-rank lattice such that all decompositions $\lat\cong\lat_1\oplus\lat_2$ with $\lat_2$ non-trivial and indecomposable have $\rank(\lat_2)=1$, then $\lat\cong \Z^n$.
\end{corollary}

\begin{proof}
    By Claim~\ref{claim:decompose}, $\lat\cong \alpha_1\Z\oplus\cdots\oplus\alpha_n \Z$. Observe that by the stability of $\lat$, we have $\alpha_i\geq 1$ for all $i$. Lastly, since $1 =\det(\lat) =  \det(\alpha_1 \Z)\cdots\det(\alpha_n \Z)$, we conclude that $\alpha_i=1$ for all $i$ and it follows that $\lat\cong\Z^n$.
\end{proof}

\subsection{The Fourier transform of the Epstein zeta function}

\label{sec:Bessel}

The \emph{modified Bessel function of the second kind} is defined as
\[
    K_{\alpha}(x) := \int_0^\infty e^{-x \cosh(t)} \cosh(\alpha t) {\rm d} t
\]
for $\alpha \in \R$ and $x > 0$.
It will be convenient for us to instead work with the function 
\[
    \Kbar_\alpha(x) := 2^{1-\alpha}x^\alpha K_\alpha(x)
    \; ,
\]
for $x >0$. We also define $\Kbar_\alpha(0) := \Gamma(\alpha)$ for $\alpha > 0$. (This is the limit as $x \to 0^+$ of $\Kbar_\alpha(x)$.)

\begin{claim}
For any $s > n/2$ and $q > 0$, let $f : \R^n \to \R$ be  $f(\vec{y}) := (\|\vec{y}\|^2 + q)^{-s}$. Then, the Fourier transform of $f$ is

\[
    \widehat{f}(\vec{w}) = \frac{\pi^{n/2} q^{n/2-s}}{\Gamma(s)} \cdot \Kbar_{s-n/2}(2\pi \sqrt{q} \|\vec{w}\|)
    \; .
\]
In particular, for any lattice $\lat \subset \R^n$, we have
\begin{equation}
    \label{eq:zeta_PSF}
     \qzf{\lat}{s}{q} =   \frac{\pi^{n/2}q^{n/2-s}}{\Gamma(s)\det(\lat)} \cdot \sum_{\vec{w} \in \lat^*} \overline{K}_{s-n/2}(2\pi \sqrt{q} \|\vec{w}\|)
    \; .
\end{equation}
\end{claim}

\begin{proof}
To see this, first notice that
\begin{equation}
\label{eq:zeta_from_theta_pointwise}
    f(\vec{y}) = \frac{\pi^s}{\Gamma(s)} \cdot \int_0^\infty \tau^{s-1}\exp(-\pi \tau q) \cdot g_\tau(\vec{y}) \intd \tau
    \; ,
\end{equation}
where $g_\tau(\vec{y}) := \exp(-\pi \tau \|\vec{y}\|^2)$,
and the equality follows by the change of variable $t = \pi q \tau + \pi \|\vec{y}\|^2\tau$ and the definition of the Gamma function. 
Recall that $\widehat{g_\tau}(\vec{w}) = \tau^{-n/2} \cdot g_{1/\tau}(\vec{w})$. By Fubini's theorem, the Fourier transform of the integral is the integral of the Fourier transform, so  we have
\[
     \widehat{f}(\vec{w}) = \frac{\pi^s}{\Gamma(s)} \cdot \int_0^\infty \tau^{s-n/2-1} \exp(-\pi \tau q) g_{1/\tau}(\vec{w}) \intd \tau
    \; .
\] 
In particular, when $\vec{w}=\vec0$, we see that $\widehat{f}(\vec0) = \pi^{n/2} q^{-(s-n/2)}\Gamma(s-n/2)/\Gamma(s)$, as needed.

Otherwise, for $\vec{w}\neq\vec0$, by a change of variable, we have
\begin{align*}
    \frac{\Gamma(s)}{\pi^s} \cdot \widehat{f}(\vec{w})
    &= (\|\vec{w}\|^2/q)^{s/2-n/4} \cdot  \int_0^\infty \tau^{s-n/2-1} \exp(-\pi \sqrt{q} \|\vec{w}\| \cdot  (\tau + 1/\tau) ) \intd \tau\\
    &= (\|\vec{w}\|^2/q)^{s/2-n/4} \cdot \int_1^\infty (\tau^{s-n/2-1} + \tau^{n/2-s-1}) \exp(-\pi \sqrt{q} \|\vec{w}\| \cdot (\tau + 1/\tau)) \intd \tau
    \; ,
\end{align*}
where in the last equality we separate the integration over the interval $[0,1]$ from $[1,\infty)$ and again use a change of variable. 
Applying the final change of variable $\tau =  \exp(t)$ and recalling the definition of $\overline{K}$ completes the calculation.
\end{proof}

The following lemma follows from the fact that the Fourier transform of $f$ is non-negative. See, e.g.,~\cite[Lemma 2.3]{regevReverseMinkowskiTheorem2017} for a proof of an analogous result for $\Theta(\lat,i\tau)$. The proof of Lemma~\ref{lemma:zeta_directSum} is essentially identical. (Alternatively, one can derive Lemma~\ref{lemma:zeta_directSum} directly from the analogous result for $\Theta(\lat,i\tau)$ by using Eq.~\eqref{eq:zeta_from_theta}.)

\begin{lemma}\label{lemma:zeta_directSum}
    For any lattice $\lat\subset\R^n$ with primitive sublattice $\lat' \subseteq \lat$, $s > n/2$, and $q\geq 0$, we have
    \[
        \qzfp{\lat}{s}{q} \leq \qzfp{(\lat/\lat') \oplus \lat'}{s}{q}
        \; ,
    \] 
    with equality if and only if $\lat$ is isomorphic to $\lat/\lat' \oplus \lat'$. 
\end{lemma}

 From this, we immediately derive the following corollary by noticing that if $\lat = \lat_1 \oplus \lat_2$ and $\lat' \subseteq \lat_2$, then $\lat/\lat' = \lat_1 \oplus (\lat_2/\lat')$.

\begin{corollary}\label{cor:zeta_directSum}
    For any lattice $\lat$ that can be written as a direct sum $\lat = \lat_1 \oplus \lat_2 \subset \R^n$, primitive sublattice $\lat' \subseteq \lat_2$, $s > n/2$, and $q\geq 0$, we have
    \[
        \qzfp{\lat}{s}{q} \leq \qzfp{\lat_1 \oplus (\lat_2/\lat') \oplus \lat'}{s}{q}
        \; ,
    \] 
    with equality if and only if $\lat_2$ is isomorphic to $\lat_2/\lat' \oplus \lat'$.
\end{corollary}

\subsection{The Laplacian}

For a twice differentiable function $f : \R^{n \times n} \to \R$, we write
\[
    \frac{\partial^2}{\partial A^2} f
    :=
    \lim_{\eps \to 0} \frac{f(\eps A) + f(-\eps A) - 2 f(0)}{\eps^2}
\]
for the second derivative of $f$ at $0$ in the $A$ direction.

The \emph{Laplacian} (at $0$) of a twice differentiable function $f$ over a  $d$-dimensional subspace $S$ of $\R^{n \times n}$ is then
\[
    \Delta_S f := \sum_{i=1}^d \frac{\partial^2}{\partial E_i^2} f
    \; ,
\]
where $E_1,\ldots, E_d$ is any orthonormal basis of $S$.
As this definition suggests, this quantity does not depend on the choice of orthonormal basis. 
We will use the simple fact that, if the Laplacian is positive, then any open subset of $S$ containing $0$ contains an $X$ such that $f(X) > f(0)$.

We write $S^{d} \subset \R^{d \times d}$ for the space of all symmetric matrices and $S_0^d \subset \R^{d \times d}$ for the space of symmetric trace-zero matrices. Notice that 
\begin{equation}
    \label{eq:laplace_minus_trace}
    \Delta_{S_0^d} f = \Delta_{S^d} f - \frac{1}{d} \cdot \frac{\partial^2}{\partial I_{d}^2}  f
    \; ,
\end{equation}
where $I_d$ is the identity matrix.

\begin{lemma}
\label{lem:laplace_computation}
Let
\[
    \mathcal{E}(A) := 
    \mathcal{E}_{\lat_1, \lat_2,s;q}(A) :=
    \qzfp{\lat_1 \oplus e^{A/2} \lat_2}{s}{q} = 
    \sum_{\substack{\vec{x} \in \lat_1,\vec{y} \in \lat_2\\(\vec{x},\vec{y})\neq \vec0}} (\|\vec{x}\|^2 + \vec{y}^T e^A \vec{y}+q)^{-s}
    \; 
\]
for lattices $\lat_1 \subset \R^{n-d}$, $\lat_2 \subset \R^d$, $s > n/2$, $q\geq 0$, and $A \in S^d$.
Then, 
\[
    \Delta_{S_0^d} \mathcal{E}
        = 
        s \cdot \frac{d-1}{d} \cdot \sum_{\substack{\vec{y} \in \lat_2\\\vec{y}\neq 0}} \big((s+1) \|\vec{y}\|^4 \cdot \qzf{\lat_1}{ s+2}{\|\vec{y}\|^2 + q}  - (d/2+1) \|\vec{y}\|^2 \cdot \qzf{\lat_1}{ s+1}{\|\vec{y}\|^2 + q}\big)
        \; .
        \]
\end{lemma}
\begin{proof}
The second derivative, derived in Appendix~\ref{app:second_derivative}, is 
\[
    \frac{\partial^2 }{\partial A^2} \mathcal{E}=s \sum_{\substack{\vec{x} \in \lat_1, \vec{y} \in \lat_2\\(\vec{x},\vec{y})\neq \vec0}} (\|\vec{x}\|^2 + \|\vec{y}\|^2+q)^{-s-1}\bigg((s+1)\frac{(\vec{y}^T A\vec{y})^2}{\|\vec{x}\|^2 + \|\vec{y}\|^2+q}-\vec{y}^T A^2 \vec{y}\bigg)\; .
\]
     Using the orthonormal basis of $S^d$ given by $E_{i,i}$ for $ i = 1,\ldots, d$ and $(E_{i,j} + E_{j,i})/\sqrt{2}$ for $1 \leq i < j \leq d$, where $E_{i,j} \in \R^{d \times d}$ is the matrix that has a $1$ in the $(i,j)$th entry and zeros elsewhere, we can compute
        \[
    \Delta_{S^d} \mathcal{E} = s\sum_{\substack{\vec{x}\in\lat_1,\vec{y}\in\lat_2\\(\vec{x},\vec{y})\neq\vec0}}(\|\vec{x}\|^2+\|\vec{y}\|^2+q)^{-s-1}
    \Big(
       (s+1)\frac{\|\vec{y}\|^4}{\|\vec{x}\|^2+\|\vec{y}\|^2+q}-\frac{d+1}{2} \cdot \|\vec{y}\|^2
    \Big)
        \]
    One can then apply Eq.~\eqref{eq:laplace_minus_trace} to calculate the Laplacian over the space $S_0^d$ of symmetric trace-zero matrices,
           \begin{align*}
        \Delta_{S_0^d} \mathcal{E}
         = & s\cdot\frac{d-1}{d}\sum_{\substack{\vec{x} \in \lat_1, \vec{y} \in \lat_2\\(\vec{x},\vec{y})\neq \vec0}}(\|\vec{x}\|^2+\|\vec{y}\|^2+q)^{-s-1}\bigg((s+1)\frac{\|\vec{y}\|^4}{\|\vec{x}\|^2+\|\vec{y}\|^2+q}-\frac{d+2}{2}\cdot\|\vec{y}\|^2\bigg)
         \; ,
\end{align*}
as needed.
\end{proof}

\section{An inequality for the Laplacian of the Epstein zeta function}

Recall that $\Kbar_\alpha(x) := 2^{1-\alpha}x^\alpha K_\alpha(x)$ for $x > 0$ and $\alpha \in \R$, and $\Kbar_\alpha(0) := \Gamma(\alpha)$ for $\alpha>0$.

\begin{lemma}
\label{lem:zeta_recurrence_inequality}
    For any $\alpha > 1$ and $x \geq 0$,
    \[
        \overline{K}_\alpha(x) \geq (\alpha-1) \overline{K}_{\alpha-1}(x) 
        \; .
    \]
\end{lemma}

\begin{proof}
    For $x = 0$, the result is trivial. For $x > 0$, by integration by parts, we have
    \begin{align*}
        K_{\alpha-1}(x) 
            &= \int_0^\infty e^{-x \cosh(t)} \cosh((\alpha-1) t) {\rm d} t \\
            &= \frac{x}{\alpha-1} \cdot \int_0^\infty  e^{-x \cosh(t)} \sinh(t) \sinh((\alpha-1) t) {\rm d} t \\
            &= \frac{x}{2(\alpha-1)} \cdot \int_0^\infty e^{-x \cosh(t)} (\cosh(\alpha t) - \cosh((\alpha - 2)t) {\rm d} t\\
            &= \frac{x}{2(\alpha-1)} \cdot ( K_{\alpha}(x) - K_{\alpha - 2}(x))
            \; .
    \end{align*}
    Rearranging, we see that
    \[
        K_{\alpha}(x) = \frac{2(\alpha - 1)}{x} \cdot K_{\alpha-1}(x) + K_{\alpha-2}(x) \geq \frac{2(\alpha - 1)}{x} \cdot K_{\alpha-1}(x)
        \; ,
    \]
    and the result follows.
\end{proof}

\begin{corollary}\label{cor:zeta_s-2_inequality}
    For any lattice $\lat \subset \R^n$, any $s > n/2$, $q > 0$, we have
        \[
        \qzf{\lat}{s+1}{q} \geq \frac{s - n/2}{qs} \cdot \qzf{\lat}{s}{q}
        \; .
    \]
\end{corollary}

\begin{proof}
Using Eq.~\eqref{eq:zeta_PSF}, we have
\begin{align*}
        \det(\lat) \cdot \qzf{\lat}{s+1}{q} 
        & = \frac{\pi^{\frac{n}{2}}q^{\frac{n}{2}-s-1}}{\Gamma(s+1)} \cdot \sum_{\vec{w}\in\lat^*}
        \overline{K}_{s+1-\frac{n}{2}}(2\pi\sqrt{q}\|\vec{w}\|)\\
        & \geq (s-\frac{n}{2}) \cdot \frac{\pi^{\frac{n}{2}}q^{\frac{n}{2}-s-1}}{\Gamma(s+1)} \cdot \sum_{\vec{w}\in\lat^*}
        \overline{K}_{s-\frac{n}{2}}(2\pi\sqrt{q}\|\vec{w}\|)\\
        & = \det(\lat) \cdot \frac{s-n/2}{qs} \cdot \qzf{\lat}{s}{q}
    \; ,
\end{align*}
where the inequality is by Lemma~\ref{lem:zeta_recurrence_inequality}, and we have again applied Eq.~\eqref{eq:zeta_PSF} in the last line.
\end{proof}

\begin{corollary}
    \label{cor:laplace_positive}
   For any lattices $\lat_1 \subset \R^{n-d}, \lat_2 \subset \R^d$ with $d \geq 2$, any $s > n/2$, and $0 \leq q \leq \frac{2s-n}{d+2} \cdot \lambda_1(\lat_2)^2$,
   \[
    \Delta_{S_0^d} \mathcal{E}_{\lat_1, \lat_2, s;q} > 0
    \; .
   \]
\end{corollary}
\begin{proof}
    By Lemma~\ref{lem:laplace_computation}, we have
    \[
    \Delta_{S_0^d} \mathcal{E}  = s \cdot  \frac{d-1}{d} \cdot \sum_{\substack{\vec{y} \in \lat_2\\\vec{y}\neq\vec0}} \big( (s+1) \|\vec{y}\|^4 \qzf{\lat_1}{s+2}{\|\vec{y}\|^2+q}- (d/2+1) \|\vec{y}\|^2 \qzf{\lat_1}{s+1}{\|\vec{y}\|^2+q} \big) 
    \; .
    \]
    Applying Corollary~\ref{cor:zeta_s-2_inequality} gives
    \begin{align*}
        \Delta_{S_0^d} \mathcal{E}
            &\geq s \cdot \frac{d-1}{d} \cdot \sum_{\substack{\vec{y} \in \lat_2\\\vec{y}\neq\vec0}} \|\vec{y}\|^2\Big( \frac{s-(n-d)/2+1}{\|\vec{y}\|^2 + q} \cdot \|\vec{y}\|^2 - (d/2+1) \Big) \qzf{\lat_1}{s+1}{\|\vec{y}\|^2+q}\\
            &> 
 s \cdot \frac{d-1}{d} \cdot \sum_{\substack{\vec{y} \in \lat_2\\\vec{y}\neq\vec0}} \|\vec{y}\|^2  \Big( \frac{s-(n-d)/2+1}{1 + q/\lambda_1(\lat_2)^2} - (d/2+1)\Big) \qzf{\lat_1}{s+1}{\|\vec{y}\|^2+q}\\
 &\geq 0
 \; ,
    \end{align*}
     as needed.
\end{proof}

\section{Proof of the main theorem}

We now prove Theorem~\ref{thm:zeta_RM_intro}.  By Item~\ref{item:contraction} of Proposition~\ref{prop:stable}, it suffices to restrict our attention to stable lattices $\lat$, as we do in the following theorem.
    
    \begin{theorem}
\label{thm:main_stable2}
    For any stable lattice $\lat \subset \R^n$, $s > n/2$, and $0 \leq q \leq \frac{2s-n}{n+2}$,
    \[
        \qzfp{\lat}{s}{q} \leq \qzfp{\Z^n}{s}{q}
        \; ,
    \]
    with equality if and only if $\lat$ is isomorphic to $\Z^n$.
\end{theorem}
\begin{proof} Since $\qzfp{\lat'}{s}{q}$ is a continuous function of $\lat'$ and the set of stable lattices with rank $n$ is compact,  $\qzfp{\lat'}{s}{q}$ achieves its maximum over the set of stable lattices. We may therefore assume that $\qzfp{\lat'}{s}{q}$ is maximized over this set at $\lat' = \lat$ and show that this implies that $\lat$ is isomorphic to $\Z^n$.
Applying Claim~\ref{claim:decompose}, we write $\lat \cong \lat_1 \oplus \lat_2$ where $\lat_2$ is indecomposable (and $\lat_1$ might be the trivial lattice $\{\vec0\}$). Note that $\lat_2$ is stable, as it is a sublattice of a stable lattice and must have determinant one  (since $\det(\lat_1) \det(\lat_2) = \det(\lat) = 1$ and $\det(\lat_i) \geq 1$). 

Then, by Corollary~\ref{cor:decomposeZn}, it suffices to show that $\lat_2$ must have rank $1$.  So, we suppose that $\rank(\lat_2)>1$ and derive a contradiction. We consider separately the case when $\lat_2$ lies on the boundary of the set of stable lattices and the case when $\lat_2$ lies in the interior of this set.

If $\lat_2$ is on the boundary of the set of stable lattices, then by Item~\ref{item:stable_quotient_boundary} of Proposition~\ref{prop:stable}, there exists a strict sublattice $\lat_3 \subset \lat_2$ such that $0<\rank(\lat_3)<\rank(\lat_2)$ where $\lat_3$ and $\lat_2/\lat_3$ are both stable. Then, applying Corollary~\ref{cor:zeta_directSum}, we have $\qzfp{\lat_1\oplus\lat_2}{s}{q}< \qzfp{\lat_1\oplus\lat_3\oplus\lat_2/\lat_3}{s}{q}$. (The inequality is strict because $\lat_2$ is indecomposable and therefore cannot be isomorphic to $\lat_3 \oplus \lat_2/\lat_3$.)  This contradicts the assumption that $\lat$ is a global maximizer of $\qzfp{\lat'}{s}{q}$ over the set of stable lattices (since $\lat_1 \oplus \lat_3 \oplus \lat_2/\lat_3$ is stable by Item~\ref{item:stable_direct_sum} of Proposition~\ref{prop:stable}).

Next, if $\lat_2$ is in the interior of the set of stable lattices, then by Corollary~\ref{cor:laplace_positive} the Laplacian of the function $\mathcal{E}(A) :=
    \qzfp{\lat_1 \oplus e^{A/2} \lat_2}{s}{q}$ over the space of trace-zero symmetric matrices $A$ is strictly positive. This implies that there exists a small perturbation $\lat''$ of $\lat_2$ that is also stable with $\qzfp{\lat_1\oplus \lat_2}{s}{q}<\qzfp{\lat_1\oplus \lat''}{s}{q}$. Since $\lat_1 \oplus \lat''$ is stable (by Item~\ref{item:stable_direct_sum} of Proposition~\ref{prop:stable}), this again contradicts the assumption that $\lat$ maximizes $\zeta'$ over the set of stable lattices.

It follows that the rank of $\lat_2$ must be $1$, and therefore $\lat\cong\Z^n$ as needed.
\end{proof}

\appendix

\section{Second derivative of the Epstein zeta function}
\label{app:second_derivative}

In this appendix, we compute the second derivative of $\mathcal{E}(A)$, which we used in the proof of Lemma~\ref{lem:laplace_computation}.

    Let $0<\eps \leq \log(2)/\|A\|$ for $\|A\|\neq 0$ (when $\|A\|=0$ the result is trivial). 
    Then, 
     \begin{align*}
       \mathcal{E}(\eps A) 
       &=  \sum_{\substack{\vec{x} \in \lat_1,\vec{y} \in \lat_2\\(\vec{x},\vec{y})\neq\vec0}} (\|\vec{x}\|^2 + \vec{y}^T e^{\eps A} \vec{y} + q)^{-s}\\
        &= \sum_{\substack{\vec{x} \in \lat_1,\vec{y} \in \lat_2\\(\vec{x},\vec{y})\neq\vec0}} (\|\vec{x}\|^2 + \|\vec{y}\|^2 + q)^{-s}\bigg(1 + \frac{\vec{y}^T(e^{\eps A} - I)\vec{y}}{\|\vec{x}\|^2 + \|\vec{y}\|^2 + q}\bigg)^{-s}\\
        &= \sum_{\substack{\vec{x} \in \lat_1,\vec{y} \in \lat_2\\(\vec{x},\vec{y})\neq\vec0}} (\|\vec{x}\|^2 + \|\vec{y}\|^2 + q)^{-s}\bigg(1-s \frac{\vec{y}^T (e^{\eps A}-I) \vec{y}}{\|\vec{x}\|^2 + \|\vec{y}\|^2 + q} + \frac{s(s+1)}{2} \cdot  \frac{(\vec{y}^T (e^{\eps A}-I) \vec{y})^2}{(\|\vec{x}\|^2 + \|\vec{y}\|^2 + q)^2} + O(\eps^3)\bigg)
       \; ,
     \end{align*}
     where the last line uses the fact that $(1+\delta)^{-s} = 1-s \delta + s(s+1) \delta^2/2 + O(\delta^3)$ for $|\delta| < 1$.
     Therefore, recalling the definition of the matrix exponential, we see that $ \mathcal{E}(\eps A)$ is equal to
      \[
       \sum_{\substack{\vec{x} \in \lat_1,\vec{y} \in \lat_2\\(\vec{x},\vec{y})\neq\vec0}} (\|\vec{x}\|^2 + \|\vec{y}\|^2 + q)^{-s}\bigg(1 -  \frac{\eps s\vec{y}^T A\vec{y}}{\|\vec{x}\|^2 + \|\vec{y}\|^2 + q} + \frac{\eps^2 s}{2 (\|\vec{x}\|^2 + \|\vec{y}\|^2 + q)}\bigg( \frac{(s+1) (\vec{y}^T A \vec{y})^2}{\|\vec{x}\|^2 + \|\vec{y}\|^2 + q} - \vec{y}^T A^2 \vec{y}\bigg) +O(\eps^3)\bigg) \; .
     \]
     
    It follows that 
    \[
        \frac{\partial^2}{\partial A^2} \mathcal{E} = s \sum_{\substack{\vec{x} \in \lat_1,\vec{y} \in \lat_2\\(\vec{x},\vec{y})\neq\vec0}} (\|\vec{x}\|^2 + \|\vec{y}\|^2+q)^{-s-1}\bigg((s+1)\frac{(\vec{y}^T A\vec{y})^2}{\|\vec{x}\|^2 + \|\vec{y}\|^2+q}-\vec{y}^T A^2 \vec{y}\bigg)
        \; ,
    \]
    as claimed.

\newcommand{\etalchar}[1]{$^{#1}$}
\def\cprime{$'$}

\end{document}